\pdfoutput=1
\documentclass{amsart}[11pt]

\usepackage{amsfonts,amssymb}
\usepackage[colorlinks=true,pdfauthor={Federico Galetto},pdftitle={Generators of truncated symmetric polynomials}]{hyperref}
\usepackage{enumitem}

\renewcommand{\geq}{\geqslant}
\renewcommand{\leq}{\leqslant}
\newcommand{\lik}{\lambda_1^{w_1}, \ldots, \lambda_k^{w_k}}
\newcommand{\rsp}{R^{\mathfrak{S}_n}}
\DeclareMathOperator{\lp}{lp}
\DeclareMathOperator{\lm}{lm}

\theoremstyle{plain}
\newtheorem{teo}{Theorem}[section]
\newtheorem{coro}[teo]{Corollary}
\newtheorem{prop}[teo]{Proposition}
\newtheorem{lemma}[teo]{Lemma}
\newtheorem*{conj}{Conjecture}
\newtheorem*{teostar}{Theorem}

\theoremstyle{definition}
\newtheorem{ex}[teo]{Example}

\title{Generators of truncated symmetric polynomials}
\author{Federico Galetto}
\address{Federico Galetto, McMaster University, 1280 Main St W, 407 Hamilton Hall, Hamilton, ON, L8S 4K1, Canada}
\email{galettof@math.mcmaster.ca}
\urladdr{http://math.galetto.org}
\date{\today}

\begin{document}
\begin{abstract}
  Adem and Reichstein introduced the ideal of truncated symmetric
  polynomials to present the permutation invariant subring in the
  cohomology of a finite product of projective spaces. Building upon
  their work, I describe a generating set of the ideal of truncated
  symmetric polynomials in arbitrary positive characteristic, and
  offer a conjecture for minimal generators.
\end{abstract}

\maketitle

\tableofcontents

\section{Introduction}
\label{sec:introduction}

Let $\mathbb{F}$ be a field, and let $R=\mathbb{F}[x_1,\ldots,x_n]$ be the polynomial ring in $n$ variables with coefficients in $\mathbb{F}$. The symmetric group $\mathfrak{S}_n$ acts on $R$ by permuting the variables. Denote by $\rsp$ the invariant subring, i.e., the ring of symmetric polynomials. The ideal of truncated symmetric polynomials in $\rsp$ is defined by
\begin{equation*}
  I_{n,d} = (x_1^{d+1},\ldots,x_n^{d+1}) R \cap \rsp.
\end{equation*}

The ideal of truncated symmetric polynomials was introduced by A.~Adem
and Z.~Reichstein \cite{MR2769082} in the following geometric context.
Let $\mathbb{C}P^d$ be the complex projective $d$-space, and let
$BU(n)$ be the classifying space of the unitary group.
The symmetric group $\mathfrak{S}_n$ acts on the $n$-fold product $(\mathbb{C}P^d)^n$ by
permuting the factors. Consider the induced action on the cohomology ring
$H^* ((\mathbb{C}P^d)^n,\mathbb{F})$ and let
$H^* ((\mathbb{C}P^d)^n,\mathbb{F})^{\mathfrak{S}_n}$ be the invariant subring.
The two authors
show there is a map $(\mathbb{C}P^d)^n \to BU(n)$ such that the
induced map on cohomology restricts to a ring epimorphism
\begin{equation*}
  H^* (BU(n),\mathbb{F}) \longrightarrow H^* ((\mathbb{C}P^d)^n,\mathbb{F})^{\mathfrak{S}_n}.
\end{equation*}
The cohomology ring of $BU(n)$ can be identified with $\rsp$, and the
kernel of this map can be identified with $I_{n,d}$. Therefore
$\rsp / I_{n,d}$ and
$H^* ((\mathbb{C}P^d)^n,\mathbb{F})^{\mathfrak{S}_n}$ are isomorphic
as ungraded rings (they become isomorphic as graded rings if the
grading on $\rsp / I_{n,d}$ is stretched out by a factor of 2). This
result was used in the same paper to compute the cohomology of the
homotopy fiber for a natural fibration over $BU(n)$.

Adem and Reichstein determined a set of generators of $I_{n,d}$ when
the characteristic of $\mathbb{F}$ is zero and when it is strictly
bigger than $(n+1)/2$. Independently, A.~Conca, C.~Krattenthaler and
J.~Watanabe \cite{MR2542141} identified generators of $I_{n,d}$
working over the complex numbers (although their methods hold over any
field of characteristic strictly bigger than $n$).  In addition, the
ideal of truncated symmetric polynomials has generated interest in
connection to other topics, such as Lefschetz properties
\cite{MR2314722} and invariants of Poincar\'e duality algebras
\cite{MR3320235}.

In this paper, I describe a generating set of $I_{n,d}$ in arbitrary characteristic. Given a partition $\lambda$, let $m_\lambda \in \rsp$ denote the monomial symmetric polynomial associated with $\lambda$. The shorthand notation $(\ldots,a^b,\ldots)$ denotes a partition with the part $a$ repeated $b$ times. The main result reads as follows.
\begin{teostar}
  Let $t=\max\{i\in\mathbb{N}\mid p^i\leq n\}$ and let
  $q_0,\ldots,q_t\in\mathbb{N}_{>0}$ be such that,
  $\forall i\in\{0,\ldots,t\}$, $n=q_i p^i+r_i$ with $0\leq
  r_i<p^i$. For every $i\in\{0,\ldots,t\}$, define an ideal of $\rsp$
  \[J_{n,d,i}=(m_{((d+1)^{p^i})},\ldots,m_{((d+q_i)^{p^i})}).\]
  Then $I_{n,d}=J_{n,d,0}+J_{n,d,1}+\ldots +J_{n,d,t}$.
\end{teostar}

The next section contains a brief review of symmetric polynomials and some computational results. In section \ref{sec:trunc-symm-polyn}, I review the algebraic results of Adem and Reichstein. The proof of the main theorem is in section \ref{sec:generators-general}. Finally, in section \ref{sec:conj-minim-gener}, I present a conjecture for a minimal generating set of the ideal of truncated symmetric polynomials.

\section{Symmetric polynomials and partitions}
Let $\mathbb{F}$ be a field of characteristic $p$. For
$n\in\mathbb{N}_{>0}$, set $R=\mathbb{F}[x_1,\ldots,x_n]$, the
polynomial ring in $n$ indeterminates over $\mathbb{F}$.  The
symmetric group $\mathfrak{S}_n$ acts on $R$ by permuting the
variables.  Let $\rsp \subseteq R$ be
the invariant subring, i.e., the ring of symmetric polynomials.

Recall that a \emph{partition} is a sequence
\[\lambda=(\lambda_1,\lambda_2,\ldots,\lambda_r,\ldots)\]
of non-negative integers in non-increasing order:
\[\lambda_1\geq\lambda_2\geq\ldots\geq\lambda_r\geq \ldots\]
and containing only finitely many nonzero terms. We identify two such
sequences which differ only by a string of zeroes at the end. The
nonzero numbers $\lambda_i$ are called the \emph{parts} of $\lambda$;
the number of parts of $\lambda$ is called the \emph{length} of
$\lambda$ and denoted $l(\lambda)$.

For a partition $\lambda$ with $l(\lambda)\leq n$, define the
\emph{monomial symmetric polynomial} on $\lambda$ to be the polynomial,
\[m_\lambda=\sum x_1^{\alpha_1} \ldots x_n^{\alpha_n}\]
summed over all distinct permutations
$\alpha=(\alpha_1,\ldots,\alpha_n)$ of
$\lambda=(\lambda_1,\ldots,\lambda_n)$. The polynomials $m_\lambda$
are clearly symmetric. Moreover, if $\mathcal{L}_{\leq n}$ denotes the
set of all partitions with length smaller than or equal to $n$, then
$\{m_\lambda\mid \lambda\in\mathcal{L}_{\leq n}\}$ is a basis of
$\rsp$ as an $\mathbb{F}$-vector space \cite[\S{}I.2,
pp. 18-19]{MR1354144}. If $l(\lambda)>n$, then $m_\lambda=0$.

\begin{ex}
  Suppose $n=3$. Then
  \begin{align*}
    m_{(1,1)}&={}x_1x_2+x_1x_3+x_2x_3,\\
    m_{(2,1)}&={}x_1^2x_2+x_1^2x_3+x_1x_2^2+x_1x_3^2+x_2^2x_3+x_2x_3^2.
  \end{align*}
\end{ex}

\begin{ex}[Power sums]
  \[p_i=x_1^i+\ldots+x_n^i=m_{(i)}\]
\end{ex}

\begin{ex}[Elementary symmetric polynomials]
  \[e_i=\sum_{1\leq j_1<\ldots <j_i\leq n} x_{j_1}\ldots
    x_{j_i}=m_{(\underbrace{{\scriptstyle 1,\ldots, 1}}_{i\text{ times}})}\]
\end{ex}

The following statement contains a multiplication formula for the
polynomials $m_\lambda$. The proof follows directly from the
definitions.

\begin{prop}
  \label{prop:multiply}
  For any $\lambda,\mu\in\mathcal{L}_{\leq n}$,
  \[m_\lambda m_\mu=\sum_{\nu\in\mathcal{L}_{\leq n}} c_\nu m_\nu,\]
  where $c_\nu$ is the number of different ways to write
  \[(\nu_1,\ldots,\nu_n)=(\alpha_1,\ldots,\alpha_n)+(\beta_1,\ldots,\beta_n)\]
  with $(\alpha_1,\ldots,\alpha_n)$ a permutation of
  $\lambda=(\lambda_1,\ldots,\lambda_n)$ and
  $(\beta_1,\ldots,\beta_n)$ a permutation of
  $\mu=(\mu_1,\ldots,\mu_n)$.
\end{prop}

\begin{ex}\label{exprod}
  Suppose $n=3$. Then
  \[m_{(1,1)} m_{(2,1)}=m_{(3,2)}+2m_{(3,1,1)}+2m_{(2,2,1)},\] since
  \begin{align*}
    (3,2,0) &= (1,1,0)+(2,1,0),\\
    (3,1,1) &= (1,1,0)+(2,0,1)=(1,0,1)+(2,1,0),\\
    (2,2,1) &= (0,1,1)+(2,1,0)=(1,0,1)+(1,2,0).
  \end{align*}
\end{ex}

Let $\lambda_1, \ldots ,\lambda_k$ be integers such that
\[\lambda_1>\lambda_2>\ldots>\lambda_k>0\]
and let $w_1, \ldots ,w_k$ be positive integers. Denote by
$(\lambda_1^{w_1}, \ldots, \lambda_k^{w_k})$ the partition $\lambda$
having $w_i$ parts equal to $\lambda_i$ for every
$i\in\{1,\ldots,k\}$. The number $w_i$ will be referred to as the
\emph{multiplicity} of $\lambda_i$ in $\lambda$. In particular,
$\lambda_1$ will be called the \emph{leading part} of $\lambda$,
denoted by $\lp(\lambda)$, and $w_1$ will be called the \emph{leading
  multiplicity} of $\lambda$, denoted by $\lm(\lambda)$. Notice
$l(\lambda)=w_1+\ldots +w_k$.

We provide here a list of formulas that will be employed later on. The
first one is a classical result by Newton (see \cite{MR1186460}).

\begin{prop}[Newton's identities] For any $s>n$,
  \[m_{(s)}=\sum_{j=1}^n (-1)^{j-1} m_{(1^j)} m_{(s-j)}.\]
\end{prop}

\begin{ex}
  Let $n=3$ and $s=4$. Then
  \begin{equation*}
    m_{(4)} = m_{(1)} m_{(3)} - m_{(1,1)} m_{(2)} + m_{(1,1,1)} m_{(1)}.
  \end{equation*}
\end{ex}

\begin{lemma}\label{comp1}
  Let $\lambda=(\lik)\in\mathcal{L}_{\leq n}$, with $k>1$. Then
  \[m_{(\lambda_1^{w_1})} m_{(\lambda_2^{w_2}, \ldots,
      \lambda_k^{w_k})}=m_{\lambda}+\sum_{\substack{\lp(\mu) >
        \lambda_1\\\lm(\mu)=w_1\\l(\mu)<l(\lambda)}} a_\mu
    m_{\mu}+\sum_{\substack{\lp(\nu) > \lambda_1\\\lm(\nu)<w_1}} b_\nu
    m_{\nu},\] for some $a_\mu,b_\nu\in\mathbb{F}$.
\end{lemma}
\begin{proof}
  Expanding the product on the left hand side in the basis of monomial
  symmetric polynomials, we obtain
  \[m_{(\lambda_1^{w_1})} m_{(\lambda_2^{w_2}, \ldots,
      \lambda_k^{w_k})} = \sum_{\theta \in \mathcal{L}_{\leq n}}
    c_\theta m_\theta,\]
  for some $c_\theta \in \mathbb{F}$. We proceed to identify the
  partitions $\theta$ for which the coefficient $c_\theta$ may be
  non-zero.
  
  Consider first the case when $\theta = \lambda$. The only way to
  obtain the partition $\lambda$ as a sum of permutations of
  $(\lambda_1^{w_1})$ and $(\lambda_2^{w_2}, \ldots, \lambda_k^{w_k})$
  is
  \[(\lambda^{w_1},0,\ldots,0)+(0^{w_1},\lambda_2^{w_2}, \ldots,
    \lambda_k^{w_k},0,\ldots,0),\]
  where, abusing notation, $0^{w_1}$ means the entry $0$ is repeated
  $w_1$ times. Thus $c_{\theta} = 1$ by Proposition
  \ref{prop:multiply}.

  Now suppose $\theta \neq \lambda$ and $c_\theta \neq 0$.  Since
  $\lambda_1>\ldots>\lambda_k>0$,
  $\lp (\theta) = \lambda_1 +\lambda_j$ for some
  $j\in\{2,\ldots,k\}$. Note that $\lambda_1 +\lambda_j > \lambda_1$,
  hence $\lp (\theta) > \lambda_1$.
  
  If there is an index $j\in\{2,\ldots,k\}$ such that $w_j\geq w_1$,
  then we have a partition
  \begin{equation*}
    \begin{split}
      \mu &= ((\lambda_1+\lambda_j)^{w_1}, \lambda_2^{w_2},\ldots,
      \lambda_j^{w_j-w_1},\ldots, \lambda_k^{w_k}) =\\
      &=(\lambda_1^{w_1},0,\ldots,0) +
      (\lambda_j^{w_1},\lambda_2^{w_2},\ldots,\lambda_j^{w_j-w_1},\ldots,\lambda_k^{w_k})
    \end{split}
  \end{equation*}
  Therefore, by Proposition \ref{prop:multiply}, the element $m_\mu$
  may appear with non-zero coefficient in our expansion.  Observe that
  $\lm (\mu) = w_1$ and $\mu$ has length
  \[w_1+w_2+\ldots +(w_j-w_1)+\ldots +w_k=\sum_{j=2}^k
    w_j<l(\lambda).\]
  Similarly, we may obtain other partitions $\mu$ with the same
  properties. We collect them all in a single summation as in the
  statement of the lemma.

  The remaining terms are of the form $b_\nu m_\nu$ where, as
  observed, $\lp (\nu) > \lambda_1$ and, in addition,
  $\lm (\nu) < w_1$. We collect them in a single summation to obtain
  the desired formula.
\end{proof}

\begin{ex}
  Let $n=6$ and $\lambda = (3,3,2,2,1)$. Then
  \begin{equation*}
    \begin{split}
      m_{(3,3)} m_{(2,2,2,1)} &= m_{(3,3,2,2,2,1)} +\\
      &\quad + m_{(5,5,2,1)} +\\
      &\quad + m_{(5,4,2,2)} + m_{(5,3,2,2,1)} + m_{(4,3,2,2,2)}.
    \end{split}
  \end{equation*}  
\end{ex}

\begin{lemma}\label{comp2}
  Let $\lambda=(\lik)\in\mathcal{L}_{\leq n}$. If $s$ is a positive
  integer such that $s+l(\lambda)\leq n$, then
  \[m_{(\lambda_1^s)} m_{\lambda}=\binom{s+w_1}{s}
    m_{(\lambda_1^{s+w_1}, \lambda_2^{w_2}, \ldots,
      \lambda_k^{w_k})}+\sum_{\substack{\lp(\mu) >
        \lambda_1\\\lm(\mu)=s\\l(\mu)=l(\lambda)}} a_\mu
    m_{\mu}+\sum_{\substack{\lp(\nu) > \lambda_1\\\lm(\nu)<s}} b_\nu
    m_{\nu},\] for some $a_\mu,b_\nu\in\mathbb{F}$.
\end{lemma}

\begin{proof}
  The proof follows from the same reasoning used to prove Lemma
  \ref{comp1}. The only aspect that requires further clarification is
  the binomial coefficient that appears in front of
  \begin{equation*}
    m_{(\lambda_1^{s+w_1}, \lambda_2^{w_2}, \ldots,
      \lambda_k^{w_k})}.
  \end{equation*}
  This polynomial arises whenever a permutation of
  $(\lambda_1^s,0,\ldots,0)$ is chosen that has $s$ of its first
  $s+w_1$ entries occupied by $\lambda_1$, and the remaining ones by
  $0$, while simultaneously a permutation of
  $(\lambda_1^{w_1}, \ldots, \lambda_k^{w_k})$ is chosen that has the
  same $s$ among its first $s+w_1$ entries occupied by $0$, and the
  remaining ones by $\lambda_1$. The number of times this occurs is
  the number of ways to choose $s$ entries among the first $s+w_1$.
\end{proof}

\begin{ex}
  Let $n=6$, $\lambda = (3^2,2,1)$ and $s=2$. Then
  \begin{equation*}
    \begin{split}
      m_{(3^2)} m_{(3^2,2,1)}
      &= 6 m_{(3^4,2,1)} +\\
      &\quad +m_{(6^2,2,1)} +\\
      &\quad +(m_{(6,5,3,1)} + m_{(6,4,3,2)} + 2 m_{(6,3^2,2,1)} +\\
      &\quad +m_{(5,4,3^2)} + 3 m_{(5,3^3,1)} + 3 m_{(4,3^3,2)}).
    \end{split}
  \end{equation*}
\end{ex}

We are interested in the case when the binomial coefficient appearing
in Lemma \ref{comp2} is invertible in $\mathbb{F}$. This is always the
case when $\mathbb{F}$ has characteristic zero. The following result
is helpful when $\mathbb{F}$ has characteristic $p>0$.

\begin{teo}[Lucas]
  Let $u$ and $v$ be non-negative integers, $p$ a prime, and
  \begin{align*}
    u &= u_t p^t + u_{t-1} p^{t-1} +\ldots + u_1 p+ u_0,\\
    v &= v_t p^t + v_{t-1} p^{t-1} +\ldots + v_1 p+ v_0,
  \end{align*}
  the base $p$ expansions of $u$ and $v$. Then
  \[ \binom{u}{v} \equiv \prod_{j=0}^t \binom{u_j}{v_j}\pmod p.\]
\end{teo}

A proof can be found in \cite[\S{}6]{MR1483922}.

\section{Truncated symmetric polynomials}
\label{sec:trunc-symm-polyn}
For a given non-negative integer $d$, consider the ideal
$(x_1^{d+1},\ldots,x_n^{d+1}) \subseteq R$.
Denote by $I_{n,d}$ the intersection
\[(x_1^{d+1},\ldots,x_n^{d+1})\cap \rsp.\]
Since $\rsp$ is a subring of $R$, $I_{n,d}$ is an ideal of $\rsp$.
Call $I_{n,d}$ the ideal of \emph{truncated symmetric polynomials}.

In this section, I will recall ideas and results previously presented
in \cite{MR2769082}.  For the convenience of the reader, all
statements and proofs will be phrased in the language and notation of
this paper.

The following proposition was recorded as a simple observation without
proof in the paper of Adem and Reichstein. I fill in the details of
the proof below.

\begin{prop}
  \label{span}\mbox{}

  \begin{enumerate}[label=\textup{(\alph*)}]
  \item\label{item:1} Let $\lambda\in\mathcal{L}_{\leq n}$. Then
    $m_\lambda \in I_{n,d}$ if and only if $\lp(\lambda)\geq d+1$.
  \item\label{item:2} The set
    \begin{equation*}
      \{m_\lambda\mid \lambda\in\mathcal{L}_{\leq n}, \lp(\lambda)\geq d+1\}
    \end{equation*}
    spans $I_{n,d}$ over $\mathbb{F}$.
  \end{enumerate}
\end{prop}
\begin{proof}
  \begin{enumerate}[label=\textup{(\alph*)}]
  \item Let $\lambda = (\lambda_1,\ldots,\lambda_n)$. If
    $\lp(\lambda)\geq d+1$, then $\lambda_1 \geq d+1$. Hence, for any
    permutation $(\alpha_1,\ldots,\alpha_n)$ of
    $(\lambda_1,\ldots,\lambda_n)$, there is some index $i$ such that
    $\alpha_i \geq d+1$ and
    $x_1^{\alpha_1} x_2^{\alpha_2} \ldots x_n^{\alpha_n} \in
    (x_1^{d+1},\ldots,x_n^{d+1})$. Therefore $m_\lambda \in I_{n,d}$.

    Vice versa, assume $m_\lambda \in I_{n,d}$. Then
    $m_\lambda \in (x_1^{d+1},\ldots,x_n^{d+1})$. Note that
    $\{x_1^{d+1},\ldots,x_n^{d+1}\}$ is the reduced Gr\"obner basis of
    the ideal $(x_1^{d+1},\ldots,x_n^{d+1})$ in the lexicographic term
    ordering on $R$. Therefore the leading term of $m_\lambda$ is
    divisible by $x_i^{d+1}$ for some $i$. Since the leading term of
    $m_\lambda$ is
    $x_1^{\lambda_1} x_2^{\lambda_2} \ldots x_n^{\lambda_n}$, it
    follows that $\lp (\lambda) = \lambda_1 \geq \lambda_i \geq d+1$.
  \item Let $0\neq g \in I_{n,d}$. Since the monomial symmetric polynomials
    form a basis of $\rsp$ over $\mathbb{F}$, $g$ can be written
    uniquely as a linear combination
    \begin{equation*}
      g = \sum_{\theta \in \mathcal{L}_{\leq n}} c_\theta m_\theta,
    \end{equation*}
    for some $c_\theta \in \mathbb{F}$. The set
      $\{\theta \in \mathcal{L}_{\leq n} \mid c_\theta \neq 0\}$
    is nonempty, finite, and totally ordered lexicographically. Let
    $\lambda$ be its maximum. Then
    $x_1^{\lambda_1} x_2^{\lambda_2} \ldots x_n^{\lambda_n}$ is the
    leading term of $g$. Reasoning as in part \ref{item:1}, deduce that
    $\lp (\lambda) \geq d+1$. Now repeat this process with
    $g - c_\lambda m_\lambda$ to show that the remaining $m_\theta$
    in the expression of $g$ also have leading part at least
    $d+1$. Therefore $g$ is in the span of the desired set.
  \end{enumerate}
\end{proof}

The next result corresponds to \cite[Lemma 5.2]{MR2769082}.

\begin{prop}\label{AR}
  Let $J_{n,d}=(m_{(d+1)},\ldots,m_{(d+n)})\subseteq\rsp$. Then $J_{n,d}$ contains
  every $m_\lambda$ with $\lp(\lambda)\geq d+1$ and
  $\lm(\lambda)\in\mathbb{F}^\times$.
\end{prop}

The last result of this section is an immediate consequence of
Proposition \ref{AR}. It is part of Adem and Reichstein's main theorem
\cite[Thm.~5.1(a)]{MR2769082}.

\begin{coro}\label{char_zero_gens}
  If $p=0$ or $n<p$, then $I_{n,d}=(m_{(d+1)},\ldots,m_{(d+n)})$.
\end{coro}
% \begin{proof}
%   If $J_{n,d} = (m_{(d+1)},\ldots,m_{(d+n)})$, then the goal is to
%   show $I_{n,d} = J_{n,d}$.  By Proposition \ref{span}\ref{item:1},
%   $I_{n,d} \supseteq J_{n,d}$.  Let $\lambda\in\mathcal{L}_{\leq n}$ with
%   $\lp (\lambda) \geq d+1$.  If $p=0$, then $\lm (\lambda)$ is
%   nonzero, hence invertible. If $n<p$, then
%   $0<\lm(\lambda)\leq l(\lambda)\leq n<p$; hence $\lm(\lambda)$ is
%   invertible. Either way, Proposition \ref{AR} implies that
%   $m_\lambda \in J_{n,d}$.  By Proposition \ref{span}\ref{item:2}, $J_{n,d}$
%   contains a generating set of $I_{n,d}$.  Thus $I_{n,d} = J_{n,d}$.
% \end{proof}

\section{Generators in positive characteristic}
\label{sec:generators-general}

As Adem and Reichstein already noticed, the polynomials
$m_{(d+1)},\ldots,m_{(d+n)}$ alone do not generate $I_{n,d}$ when
$n\geq p$. For $n\leq 2p-1$, the two authors showed that taking
$m_{((d+1)^p)}$ in addition to $m_{(d+1)},\ldots,m_{(d+n)}$ is enough
to generate $I_{n,d}$ \cite[Thm.~5.1(b)]{MR2769082}.  The goal of this
section is to describe a set of generators of $I_{n,d}$ in arbitrary
positive characteristic. Throughout this section, $p>0$ will be
assumed.

The following is a generalization of Newton's identities.

\begin{lemma}\label{newton-like}
  Consider $h\leq n$ and let
  $q\in\mathbb{N}_{>0}$ be such that $n=q h+r$ with $0\leq r<
  h$. For every $s\in\mathbb{N}_{>0}$,
  \[\sum_{j=1}^q (-1)^{j-1} m_{(1^{jh})}
    m_{((s+q+1-j)^{h})}=m_{((s+q+1)^{h})}
    +\sum_{\substack{\lp(\mu) > s+1\\\lm(\mu)<h}} a_\mu m_{\mu},\]
  for some $a_\mu\in\mathbb{F}$.
\end{lemma}
\begin{proof}
  When $j\in\{1,\ldots,q-1\}$, the product
  $m_{(1^{jh})} m_{((s+q+1-j)^{h})}$ expands to
  \[m_{((s+q+2-j)^{h},1^{(j-1)h})}+m_{((s+q+1-j)^{h},1^{jh})}+\sum_{\substack{\lp(\mu)=
        s+q+2-j\\\lm(\mu)<h}} a_{\mu,j} m_{\mu},\]
  for some $a_{\mu,j}\in\mathbb{F}$.  When $j=q$, the same product
  becomes $m_{(1^{qh})} m_{((s+1)^{h})}$ which expands to
  \[m_{((s+2)^{h},1^{(q-1)h})}+\sum_{\substack{\lp(\mu) =
        s+2\\\lm(\mu)<h}} a_{\mu,q} m_{\mu},\]
  for some $a_{\mu,q}\in\mathbb{F}$. Summing over $j$ with alternating
  signs, all monomial symmetric polynomials with leading multiplicity
  $h$ cancel out except for $m_{((s+q+1)^{h})}$ which occurs when
  $j=1$.  The remaining terms can be combined in a single summation to
  give the formula above.
\end{proof}

\begin{prop}\label{iplusj}
  Let $i\in\mathbb{N}$ be such that $p^i\leq n$ and let
  $q\in\mathbb{N}_{>0}$ be such that $n=qp^i+r$ with $0\leq r<
  p^i$. Let $K$ be an ideal of $\rsp$ such that
  \[K \supseteq \{m_\lambda\mid \lambda \in \mathcal{L}_{\leq n},
    \lp(\lambda)\geq d+1, \lm(\lambda) <p^i\}.\]
  If $J = (m_{((d+1)^{p^i})},\ldots,m_{((d+q)^{p^i})})$, then $J+K$
  contains every $m_\lambda$ with $\lp(\lambda)\geq d+1$ and
  $\lm(\lambda)\leq p^i$.
\end{prop}
\begin{proof}
  If $\lambda\in\mathcal{L}_{\leq n}$ has $\lp(\lambda)\geq d+1$ and
  $\lm(\lambda) <p^i$, then, by the hypothesis,
  $m_\lambda \in K\subseteq J+K$.
	
  Suppose $\lp (\lambda) \geq d+1$ and $\lm(\lambda) =p^i$. The proof
  that $m_\lambda\in J+K$ is by induction on $l(\lambda)$. For the
  base case, assume $l(\lambda) =p^i$, i.e.,
  $m_\lambda =m_{(\lambda_1^{p^i})}$ for some $\lambda_1\geq d+1$. By
  definition, $m_{(\lambda_1^{p^i})} \in J\subseteq J+K$, when
  $d+1 \leq \lambda_1 \leq d+q$.  If $\lambda_1 > d+q$, then we may
  write $\lambda_1 = s+q+1$, for some $s\geq d$.  Applying Lemma
  \ref{newton-like} with $h=p^i$,
  \[\sum_{j=1}^q (-1)^{j-1} m_{(1^{jp^i})}
    m_{((s+q+1-j)^{p^i})}=m_{((s+q+1)^{p^i})}
    +\sum_{\substack{\lp(\mu) > s+1\\\lm(\mu)<p^i}} a_\mu m_{\mu},\]
  for some $a_\mu\in\mathbb{F}$.  The summation on the right is in
  $K\subseteq J+K$ by the hypothesis.  Proceeding by induction on $s$,
  the left hand side belongs to $J+K$.  Hence
  $m_{((s+q+1)^{p^i})}\in J+K$.
	
  For the inductive step, assume
  $\lambda=(\lambda_1^{p^i}, \lambda_2^{w_2}, \ldots,
  \lambda_k^{w_k})$
  with $l(\lambda) >p^i$ and $\lambda_1\geq d+1$. By Lemma
  \ref{comp1},
  \[m_{(\lambda_1^{p^i})} m_{(\lambda_2^{w_2}, \ldots,
      \lambda_k^{w_k})}=m_\lambda +\sum_{\substack{\lp(\mu) >
        \lambda_1\\\lm(\mu)=p^i\\l(\mu)<l(\lambda)}} a_\mu
    m_{\mu}+\sum_{\substack{\lp(\nu) > \lambda_1\\\lm(\nu)<p^i}} b_\nu
    m_{\nu},\]
  for some $a_\mu,b_\nu\in\mathbb{F}$. By the base case,
  $m_{(\lambda_1^{p^i})}\in J+K$. On the right hand side, the first
  summation belongs to $J+K$ by the inductive hypothesis and the
  second one belongs to $K\subseteq J+K$ by the hypothesis. Thus
  $m_\lambda\in J+K$.
\end{proof}

We are now ready to exhibit a set of generators for the ideal
$I_{n,d}$.

\begin{teo}\label{gens}
  Let $t=\max\{i\in\mathbb{N}\mid p^i\leq n\}$ and let
  $q_0,\ldots,q_t\in\mathbb{N}_{>0}$ be such that,
  $\forall i\in\{0,\ldots,t\}$, $n=q_i p^i+r_i$ with $0\leq
  r_i<p^i$. For every $i\in\{0,\ldots,t\}$, define an ideal of $\rsp$
  \[J_{n,d,i}=(m_{((d+1)^{p^i})},\ldots,m_{((d+q_i)^{p^i})}).\]
  Then $I_{n,d}=J_{n,d,0}+J_{n,d,1}+\ldots +J_{n,d,t}$.
\end{teo}

\begin{ex}\label{ex271}
  Let $p=2$, $n=7$, and $d=1$. Then
  \begin{align*}
    &J_{7,1,0} = (m_{(2)},m_{(3)},m_{(4)},m_{(5)},m_{(6)},m_{(7)},m_{(8)}),\\
    &J_{7,1,1} = (m_{(2^2)},m_{(3^2)},m_{(4^2)}),\\
    &J_{7,1,2} = (m_{(2^4)}),
  \end{align*}
  and $I_{7,1} = J_{7,1,0} + J_{7,1,1} + J_{7,1,2}$.
\end{ex}

\begin{proof}
  Define auxiliary ideals by setting
  \[K_i :=J_{n,d,0}+ J_{n,d,1} +\ldots +J_{n,d,i},\]
  for $i\in\{0,\ldots,t\}$. With this notation, the thesis becomes
  $I_{n,d} = K_t$.  It is clear that
  $K_0 \subseteq K_1 \subseteq\ldots\subseteq K_t$.
  Moreover, $K_t \subseteq I_{n,d}$ by Proposition
  \ref{span}\ref{item:1}.  In light of Proposition
  \ref{span}\ref{item:2}, the inclusion $I_{n,d}\subseteq K_t$
  will follow if one can show that $K_t$ contains every polynomial
  $m_\lambda$ with $\lp(\lambda)\geq d+1$.

  \underline{Claim}: $\forall i\in\{0,\ldots,t\}$, the inclusion
  \begin{equation*}
    K_i \supseteq \{m_\lambda\mid \lambda \in \mathcal{L}_{\leq n},
    \lp(\lambda)\geq d+1, \lm(\lambda) <p^{i+1}\}
  \end{equation*}
  holds.

  The claim will be proven by induction on $i$. For $i=0$, note that
  \begin{equation*}
    K_0 =J_{n,d,0}=(m_{(d+1)},\ldots,m_{(d+n)}).
  \end{equation*}
  Hence the claim is an immediate consequence of Proposition \ref{AR}.

  Assume $i>0$. By the inductive hypothesis,
  \[K_{i-1}\supseteq \{m_\lambda\mid \lambda \in \mathcal{L}_{\leq
      n}, \lp(\lambda)\geq d+1, \lm(\lambda) <p^{i}\}.\]
  By definition, $K_{i}=J_{n,d,i} +K_{i-1}$. Hence, by
  Proposition \ref{iplusj},
  \begin{equation}\label{star}
    K_{i}\supseteq \{m_\lambda\mid \lambda \in \mathcal{L}_{\leq n},
    \lp(\lambda)\geq d+1, \lm(\lambda) \leq p^{i}\}.
  \end{equation}

  It remains to show $K_{i}$ contains all monomial symmetric
  polynomials $m_\lambda$ with $p^i<\lm(\lambda)<p^{i+1}$. Let
  $\lambda=(\lik)\in\mathcal{L}_{\leq n}$ with $\lambda_1\geq d+1$ and
  $p^i<w_1<p^{i+1}$. The base $p$ expansion of $w_1$ is
  \[w_1 = h_i p^i + h_{i-1} p^{i-1} +\ldots + h_1 p+ h_0,\]
  where $0\leq h_j<p$, $\forall j\in\{0,\ldots,i\}$, and $h_i\neq 0$.
  Set $z= h_{i-1} p^{i-1} +\ldots + h_0$, so that $w_1=h_i p^i+z$. By
  Lemma \ref{comp2},
  \[m_{(\lambda_1^{p^i})}
    m_{(\lambda_1^{(h_i-1)p^i+z},\lambda_2^{w_2},\ldots,\lambda_k^{w_k})}=\binom{w_1}{p^i}
    m_\lambda+\sum_{\substack{\lp(\mu) >
        \lambda_1\\\lm(\mu)=p^i\\l(\mu)=l(\lambda)}} a_\mu
    m_{\mu}+\sum_{\substack{\lp(\nu) > \lambda_1\\\lm(\nu)<p^i}} b_\nu
    m_{\nu},\]
  for some $a_\mu,b_\nu\in\mathbb{F}$. By (\ref{star}),
  $m_{(\lambda_1^{p^i})}$ and the two summations on the right hand
  side all belong to $K_{i}$. Moreover, by Lucas' theorem,
  \[ \binom{w_1}{p^i}\equiv \binom{h_i}{1}
    \prod_{j=0}^{i-1}\binom{h_j}{0} \equiv h_i\pmod p.\]
  Since $0<h_i<p$, $\binom{w_1}{p^i}\in\mathbb{F}^\times$; therefore
  $m_\lambda\in K_{i}$. This concludes the proof of the claim.

  When $i=t$, the claim gives
  \[K_{t}\supseteq \{m_\lambda\mid \lambda \in \mathcal{L}_{\leq n},
    \lp(\lambda)\geq d+1, \lm(\lambda) <p^{t+1}\}.\]
  If $\lambda\in\mathcal{L}_{\leq n}$ and $\lp(\lambda)\geq d+1$, the
  inequality $\lm(\lambda)\leq l(\lambda)\leq n<p^{t+1}$ implies
  $m_\lambda\in K_{t}$. Therefore $I_{n,d}=K_t$.
\end{proof}

\section{A conjecture for minimal generators}
\label{sec:conj-minim-gener}
In this section, I consider the problem of describing minimal
generators of the ideal $I_{n,d}$. When $p=0$ or $p>n!$, a solution
was already given by Adem and Reichstein (see
\cite[Thm.~6.1(a)]{MR2769082} and \cite[Lemma~6.2(c)]{MR2769082});
namely, the generators of Corollary \ref{char_zero_gens} are minimal.

Theorem \ref{gens} describes a set of generators of
$I_{n,d}$ for arbitrary $p>0$. In general, this set is not minimal.

\begin{ex}
  Let $p=2$, $n=7$ and $d=1$. Refer to Example \ref{ex271} for a
  generating set of $I_{7,1}$. Notice that
  \begin{equation*}
    (m_{(2)})^2 = \left( \sum_{i=1}^7 x_i^2 \right)^2 =
    \sum_{i=1}^7 x_i^4 = m_{(4)},
  \end{equation*}
  making $m_{(4)}$ redundant as a generator. Similarly, $m_{(6)}$ and
  $m_{(8)}$ can be discarded because $m_{(6)} = (m_{(3)})^2$ and
  $m_{(8)} = (m_{(4)})^2$.

  A similar computation yields
  \begin{equation*}
    (m_{(2^2)})^2 = \left( \sum_{1\leq i< j\leq 7} x_i^2 x_j^2 \right)^2 =
    \sum_{1\leq i<j \leq 7} x_i^4 x_j^4 = m_{(4^2)}.
  \end{equation*}
  Therefore $m_{(4^2)}$ is also redundant.
\end{ex}

In general, when $p>0$, we have
\begin{equation}\label{relation}
  (m_{(s^{k})})^p=m_{((ps)^{k})}.
\end{equation}
This observation allows to trim the generating set of Theorem
\ref{gens}, leading to the following.

\begin{conj}
  Let $t=\max\{i\in\mathbb{N}\mid p^i\leq n\}$ and let
  $q_0,\ldots,q_t\in\mathbb{N}_{>0}$ be such that,
  $\forall i\in\{0,\ldots,t\}$, $n=q_i p^i+r_i$ with $0\leq
  r_i<p^i$. For every $i\in \{0,\ldots,t\}$, define the set
  \[\mathcal{J}_{n,d,i} = \left\{m_{((d+h)^{p^i})} \middle\arrowvert
      h\in\{1,\ldots,q_i\}\text{ and } \forall
      j\in\{1,\ldots,h\}\ d+h\neq p(d+j)\right\}.\]
  Then
  $\mathcal{J}_{n,d,0} \cup \mathcal{J}_{n,d,1} \cup \ldots \cup
  \mathcal{J}_{n,d,t}$ is a minimal set of generators of $I_{n,d}$.
\end{conj}

Essentially, this conjecture claims that the only possible relations
among the generators of Theorem \ref{gens} are of the kind given in
equation \eqref{relation}.  For small values of $p$, $n$ and $d$, the
conjecture has been verified using the software Macaulay2
\cite{M2}.

\section{Acknowledgements}
\label{sec:acknowledgements}

I am indebted to the anonymous referee, whose useful comments and
suggestions helped improve the quality of this manuscript.  While
working on this project, I was partially supported by an NSERC grant.

\end{document}